\documentclass[11pt,a4paper]{article}

	\usepackage{amssymb}
  \usepackage[colorlinks]{hyperref}
	\usepackage{kbordermatrix}
	\usepackage{graphicx}
	\usepackage{amsmath,amsthm}
	\usepackage[numbers]{natbib}
	\newcommand{\Dutchvon}[2]{#2}

  \renewcommand{\phi}{\varphi}

	\newcommand{\mathds}{\mathbb}
	\newcommand{\Q}{ \mathds{Q} }

	\newcommand{\Z}{ \mathds{Z} }

	\newcommand{\group}{G}
	
	\DeclareMathOperator{\GF}{GF}
	\newcommand{\ring}{R}
	\DeclareMathOperator{\rank}{rk}
	\newcommand{\parf}{\ensuremath{\mathbb{P}}} 

	\renewcommand{\rank}{\mathrm{rk}}

	\newcommand{\ignore}[1]{}

	\DeclareMathOperator{\closure}{cl}

	\let\Oldsetminus\setminus
	\renewcommand{\setminus}{\ensuremath{-}}

	\newcommand{\delete}{\ensuremath{\!\Oldsetminus\!}}
	\newcommand{\contract}{\ensuremath{/}}

	\newtheorem{theorem}{Theorem}[section]
	
	\newtheorem{lemma}[theorem]{Lemma}
	\newtheorem{proposition}[theorem]{Proposition}
	\newtheorem{conjecture}[theorem]{Conjecture}
	
	\theoremstyle{definition}
	\newtheorem{definition}[theorem]{Definition}

	\newenvironment{claimenv}{\list{}{\rightmargin0pt\leftmargin10pt\topsep0pt}\item[]}{\endlist}

\begin{document}
\title{An obstacle to a decomposition theorem for near-regular matroids\thanks{Parts of this research have appeared in the third author's PhD thesis \cite{vZ09}. The research of all authors was partially supported by a grant from the Marsden Fund of New Zealand. The first author was also supported by a FRST Science \& Technology post-doctoral fellowship. The third author was also supported by NWO, grant 613.000.561.}}
\author{Dillon Mayhew\thanks{School of Mathematics, Statistics and Operations Research, Victoria University of Wellington, New Zealand. E-mail: \url{Dillon.Mayhew@msor.vuw.ac.nz}, \url{Geoff.Whittle@msor.vuw.ac.nz}} \and Geoff Whittle\footnotemark[2] \and Stefan H. M. van Zwam\thanks{Centrum Wiskunde en Informatica, Postbus 94079, 1090 GB Amsterdam, The Netherlands. E-mail: \url{Stefan.van.Zwam@cwi.nl}}}

\maketitle

\abstract{
	Seymour's Decomposition Theorem for regular matroids states that any matroid representable over both $\GF(2)$ and $\GF(3)$ can be obtained from matroids that are graphic, cographic, or isomorphic to $R_{10}$ by $1$-, $2$-, and $3$-sums. It is hoped that similar characterizations hold for other classes of matroids, notably for the class of near-regular matroids. Suppose that all near-regular matroids can be obtained from matroids that belong to a few basic classes through $k$-sums. Also suppose that these basic classes are such that, whenever a class contains all graphic matroids, it does not contain all cographic matroids. We show that in that case 3-sums will not suffice.
}
%
%

\section{Introduction}
A regular matroid is a matroid representable over every field. Much is known about this class, the deepest result being Seymour's Decomposition Theorem:

\begin{theorem}[{Seymour~\cite{Sey80}}]\label{thm:regdec}
  Let $M$ be a regular matroid. Then $M$ can be obtained from matroids that are graphic, cographic, or equal to $R_{10}$ through $1$-, $2$-, and $3$-sums.
\end{theorem}

A class ${\cal C}$ of matroids is \emph{polynomial-time recognizable} if there exists an algorithm that decides, for any matroid $M$, in time $f(|E(M)|, \tau)$ whether or not $M \in {\cal C}$, where $\tau$ is the time of one rank evaluation, and $f(x,y)$ a polynomial. Seymour~\cite{Sey81} showed that the class of graphic matroids is polynomial-time recognizable. Also every finite class is polynomial-time recognizable. Using these facts Truemper~\cite{Tru82} (see also Schrijver \cite[Chapter 20]{Sch86}) showed the following:

\begin{theorem}\label{cor:regpoly}
  The class of regular matroids is polynomial-time recognizable.
\end{theorem}

A \emph{near-regular matroid} is a matroid representable over every field, except possibly $\GF(2)$. Near-regular matroids were introduced by Whittle \cite{Whi95,Whi97}. The following is one of his results:
\begin{theorem}[{Whittle~\cite{Whi97}}]
  Let $M$ be a matroid. The following are equivalent:
  \begin{enumerate}
    \item $M$ is representable over $\GF(3)$, $\GF(4)$, and $\GF(5)$;
    \item $M$ is representable over $\Q(\alpha)$ by a totally near-unimodular matrix;
    \item $M$ is near-regular.
  \end{enumerate}
\end{theorem}

  In this theorem $\alpha$ is an indeterminate. A \emph{totally near-unimodular matrix} is a matrix over $\Q(\alpha)$ such that the determinant of every square submatrix is either zero or equal to $(-1)^s \alpha^i (1-\alpha)^j$ for some $s,i,j \in \Z$. Whittle~\cite{Whi97,Whi05} wondered if an analogue of Theorem~\ref{thm:regdec} would hold for the class of near-regular matroids. The following conjecture was made:

\begin{conjecture}\label{con:nregdecompfalse}
  Let $M$ be a near-regular matroid. Then $M$ can be obtained from matroids that are signed-graphic, their duals, or members of some finite set through $1$-, $2$-, and $3$-sums.
\end{conjecture}

A matroid is \emph{signed-graphic} if it can be represented by a $\GF(3)$-matrix with at most two nonzero entries in each column (see Zaslavsky~\cite{Zas82,Zas82err} for more on these matroids). One difference with the regular case is that not every signed-graphic matroid is near-regular.

Several people have made an effort to understand the structure of near-regular matroids. Oxley et al.~\cite{OVW98} studied maximum-sized near-regular matroids. Hlin{\v{e}}n{\'{y}}~\cite{Hli04} and Pendavingh~\cite{Pen04} have both written software to investigate all 3-connected near-regular matroids up to a certain size. 
Pagano~\cite{Pag98} studied signed-graphic near-regular matroids, and Pendavingh and Van Zwam~\cite{PZ09graphic} studied a closely related class of matroids which they call near-regular-graphic.

Despite these efforts, an analogue to Theorem~\ref{thm:regdec} is still not in sight. In this paper we record an obstacle we found, that will have to be taken into account in any structure theorem. Our result is the following:
\begin{theorem}\label{thm:obstruct}
  Let $G_1, G_2$ be graphs. There exists an internally 4-connected near-regular matroid $M$ having both $M(G_1)$ and $M(G_2)^*$ as a minor.
\end{theorem}

From this, and the fact that not all cographic matroids are signed-graphic, it follows that Conjecture~\ref{con:nregdecompfalse} is false. More generally, suppose we want to find a decomposition theorem for near-regular matroids, such that each basic class that contains all graphic matroids, does not contain all cographic matroids. Theorem~\ref{thm:obstruct} implies that such a characterization must employ at least 4-sums.

The paper is organized as follows. In Section~\ref{sec:prelim} we give some preliminary definitions. In Section \ref{sec:genparcon} we prove a lemma that shows how generalized parallel connection can preserve representability over a partial field. In Section~\ref{sec:4sums} we prove Theorem~\ref{thm:obstruct}. 
We conclude in Section~\ref{sec:conjectures} with some updated conjectures.

Throughout this paper we assume familiarity with matroid theory as set out in Oxley~\cite{oxley}.

\section{Preliminaries}\label{sec:prelim}
\subsection{Connectivity}
In addition to the usual definitions of connectivity and separations (see Oxley \cite[Chapter 8]{oxley}) we say a partition $(A,B)$ of the ground set of a matroid is \emph{$k$-separating} if $\rank_M(A) + \rank_M(B) - \rank(M) < k$. Recall that $(A,B)$ is a \emph{$k$-separation} if it is $k$-separating and $\min\{|A|,|B|\} \geq k$.

\begin{definition}
  A matroid is \emph{internally 4-connected} if it is 3-connected and $\min\{|X|,|Y|\} = 3$ for every 3-separation $(X,Y)$.
\end{definition}

This notion of connectivity is useful in our context. For instance, Theorem~\ref{thm:regdec} can be rephrased as follows:
\begin{theorem}
  Let $M$ be an internally 4-connected regular matroid. Then $M$ is graphic, cographic, or equal to $R_{10}$.
\end{theorem}

Intuitively, separations $(X,Y)$ where both $|X|$ and $|Y|$ are big should give rise to a decomposition into smaller matroids.
\begin{definition}
  Let $M$ be a matroid, and $N$ a minor of $M$. Let $(X',Y')$ be a $k$-separation of $N$. We say that $(X',Y')$ is \emph{induced} in $M$ if $M$ has a $k$-separation $(X,Y)$ such that $X' \subseteq X$ and $Y'\subseteq Y$. 
\end{definition}

At several points we will use the following easy fact:
\begin{lemma}\label{lem:minorksep}
	Let $M$ be a matroid, let $N$ be a minor of $M$, and let $(A,B)$ be a $k$-separating partition of $E(M)$. Then $(A\cap E(N), B\cap E(N))$ is $k$-separating in $N$.
\end{lemma}

Note that $(A\cap E(N), B\cap E(N))$ need not be \emph{exactly} $k$-separating.

\subsection{Partial fields}
Our main tool in the proof of Theorem~\ref{thm:obstruct} is useful outside the scope of this paper. Hence we have stated it in the general framework of partial fields. For that purpose we need a few definitions. More on the theory of partial fields can be found in Semple and Whittle~\cite{SW96} and in Pendavingh and Van Zwam~\cite{PZ08lift,PZ08conf}.
\begin{definition}
  A \emph{partial field} is a pair $(\ring, \group)$, where $\ring$ is a commutative ring with identity, and $\group$ is a subgroup of the group of units of $\ring$ such that $-1 \in \group$.
\end{definition}

For example, the near-regular partial field is $\left(\Q(\alpha), \langle -1,\alpha, 1-\alpha\rangle\right)$, where $\langle S \rangle$ denotes the multiplicative group generated by $S$. For $\parf = (R,G)$, we abbreviate $p \in G\cup\{0\}$ to $p\in\parf$.

We will adopt the convention that matrices have labelled rows and columns, so an $X\times Y$ matrix $A$ is a matrix whose rows are labelled by the (ordered) set $X$ and whose columns are labelled by the (ordered) set $Y$. The identity matrix with rows and columns labelled by $X$ will be denoted by $I_X$. We will omit the subscript if it can be deduced from the context.

Let $A$ be an $X\times Y$ matrix. If $X'\subseteq X$ and $Y'\subseteq Y$ then we denote the submatrix of $A$ indexed by $X'$ and $Y'$ by $A[X',Y']$. If $Z\subseteq X\cup Y$ then we write $A[Z] := A[X\cap Z, Y\cap Z]$. If $A$ is an $X\times Y$ matrix, where $X\cap Y = \emptyset$, then we denote by $[I\  A]$ the $X\times (X\cup Y)$ matrix obtained from $A$ by prepending the identity matrix $I_X$.

\begin{definition}
  Let $\parf := (\ring, \group)$ be a partial field, and let $A$ be a matrix with entries in $\ring$. Then $A$ is a \emph{$\parf$-matrix} if, for every square submatrix $A'$ of $A$, either $\det(A') = 0$ or $\det(A') \in \group$.
\end{definition}

\begin{theorem}
  Let $\parf$ be a partial field, let $A$ be an $X\times Y$ $\parf$-matrix for disjoint sets $X$ and $Y$, let $E := X\cup Y$, and let $A' := [I\  A]$. If ${\cal B} = \{B \subseteq E : |B| = |X|, \det(A'[X,B]) \neq 0\}$, then ${\cal B}$ is the set of bases of a matroid.
\end{theorem}

We denote this matroid by $M[I\ A]$.

\subsection{Pivoting}
Let $A$ be an $X\times Y$ $\parf$-matrix. Then $X$ is a basis of $M[I\ A]$. We say that $X$ is the \emph{displayed} basis. Pivoting in the matrix allows us to change the basis that is displayed. Roughly speaking a pivot in $A$ consists of row reduction applied to $[I\ A]$, followed by a column exchange. The precise definition is as follows:

\begin{definition}\label{def:pivot}
	Let $A$ be an $X\times Y$ matrix over a ring $\ring$, and let $x \in X, y \in Y$ be such that $A_{xy} \in \ring^*$. Then $A^{xy}$ is the $(X\setminus x)\cup y \times (Y\setminus y)\cup x$ matrix with entries
	\begin{align*}
	  (A^{xy})_{uv} = \left\{ \begin{array}{ll}
	    (A_{xy})^{-1} \quad & \textrm{if } uv = yx\\
	    (A_{xy})^{-1} A_{xv} & \textrm{if } u = y, v\neq x\\
	    -A_{uy} (A_{xy})^{-1} & \textrm{if } v = x, u \neq y\\
	    A_{uv} - A_{uy} (A_{xy})^{-1} A_{xv} & \textrm{otherwise.}
	  \end{array}\right.
	\end{align*}
\end{definition}

We say that $A^{xy}$ was obtained from $A$ by \emph{pivoting}. Slightly less opaquely, if
\begin{align*}
			A = \kbordermatrix{ & y & Y'\\
	                     x  & a & \: c \: \\
	                     X' & b & \phantom{\dfrac{1}{1}}D\phantom{\dfrac{1}{1}}}
\end{align*}
then
\begin{align*}
	 A^{xy} = \kbordermatrix{ & x & Y'\\
	                     y  & a^{-1} & \: a^{-1}c \: \\
	                      X' & -ba^{-1} & \phantom{\dfrac{1}{1}}D-ba^{-1}c}.
\end{align*}

As Semple and Whittle\cite{SW96} proved, pivoting maps $\parf$-matrices to $\parf$-matrices:
\begin{proposition}\label{prop:pivotproper}
  Let $A$ be an $X\times Y$  $\parf$-matrix, and let $x \in X, y \in Y$ be such that $A_{xy} \neq 0$. Then $A^{xy}$ is a $\parf$-matrix, and $M[I\ A] = M[I\ A^{xy}]$.
\end{proposition}

Semple and Whittle also showed that pivots can be used to compute determinants of $\parf$-matrices:
\begin{lemma}\label{lem:detpivot}
  Let $\parf$ be a partial field, and let $A$ be an $X\times Y$ $\parf$-matrix with $|X| = |Y|$. If $x\in X, y\in Y$ is such that $A_{xy} \neq 0$ then
  \begin{align*}
  \det(A)= (-1)^{x+y} A_{xy} \det(A^{xy}[X\setminus x, Y\setminus y]).
  \end{align*}
\end{lemma}

\section{Generalized parallel connection}\label{sec:genparcon}

Recall the generalized parallel connection of two matroids $M_1$, $M_2$ along a common restriction $N$, denoted by $P_N(M_1,M_2)$. This construction was introduced by Brylawski \cite{Bry75} (see also Oxley~\cite[Section~12.4]{oxley}). Brylawski proved that representability over a field can be preserved under generalized parallel connection, provided that the representations of the common minor are identical. Lee~\cite{Lee90} generalized Brylawski's result to matroids representable over a field such that all subdeterminants are in a multiplicatively closed set. We generalize Brylawski's result further to matroids representable over a partial field, as follows.

\begin{theorem}\label{thm:parcon}
  Suppose $A_1$, $A_2$ are $\parf$-matrices with the following structure:
  \begin{align*}
    A_1 = \kbordermatrix{ & Y_1 & Y\\
                       X_1 & D_1' & 0\\
                       X & D_1 & D_X
                       }, \qquad
    A_2 = \kbordermatrix{ & Y & Y_2\\
                       X & D_X & D_2\\
                       X_2 & 0 & D_2'
                      },
  \end{align*}
  where $X,Y,X_1,Y_1,X_2,Y_2$ are pairwise disjoint sets. If $X\cup Y$ is a modular flat of $M[I\  A_1]$ then
  \begin{align*}
    A := \kbordermatrix{
                  & Y_1 & Y   & Y_2\\
              X_1 & D_1'& 0   & 0\\
              X   & D_1 & D_X & D_2\\
              X_2 & 0   & 0   & D_2'
        }
  \end{align*}
  is a $\parf$-matrix. Moreover, if $M_1 = M[I\  A_1]$ and $M_2 = M[I\  A_2]$, then $M[I\  A] = P_{N}(M_1,M_2)$, where $N = M[I\ D_X]$.\index{generalized parallel connection}
\end{theorem}

The main difficulty is to show that $A$ is a $\parf$-matrix. 
To prove this we will use a result known as the modular short-circuit axiom \citep[Theorem 3.11]{Bry75}. We use Oxley's formulation~\cite[Theorem 6.9.9]{oxley}, and refer to that book for a proof.

\begin{lemma}\label{lem:modshortcircuit}
  Let $M$ be a matroid and $X \subseteq E$ nonempty. The following statements are equivalent:
  \begin{enumerate}
    \item\label{it:msc1} $X$ is a modular flat of $M$;
    \item\label{it:msc2} For every circuit $C$ such that $C\setminus X \neq \emptyset$, there is an element $x \in X$ such that $(C\setminus X) \cup x$ is dependent.
    \item\label{it:msc3} For every circuit $C$, and for every $e \in C\setminus X$, there are an $f \in X$ and a circuit $C'$ such that $e \in C'$ and $C' \subseteq (C\setminus X) \cup f$.
  \end{enumerate}
\end{lemma}

The following is an extension of Proposition 4.1.2 in \cite{Bry75} to partial fields. Note that Brylawski proves an ``if and only if'' statement, whereas we only state the ``only if'' direction.

\begin{lemma}\label{lem:modflatparallel}
  Let $M = (E, \mathcal{I})$ be a matroid, and $X$ a modular flat of $M$. Suppose $B_X$ is a basis for $M|X$, and $B\supseteq B_X$ a basis of $M$. Suppose $A$ is a $B\times (E\setminus B)$ $\parf$-matrix such that $M = M[I\: A]$. Then every column of $A[B_X, E\setminus (B\cup X)]$ is a $\parf$-multiple of a column of $[I\: A[B_X,X\setminus B]]$.
\end{lemma}

\begin{proof}[Proof of Lemma~\ref{lem:modflatparallel}]
  Let $M$, $X$, $B_X$, $B$, $A$ be as in the lemma, so
  \begin{align*}
    A = \kbordermatrix{ & E \setminus (B\cup X) & X\setminus B\\
                B\setminus B_X & D' & 0\\
                B_X    &   D & D_{B_X}\\
                }.
  \end{align*}
  Let $v \in E\setminus (B\cup X)$, and let $C$ be the $B$-fundamental circuit containing $v$. If $C\cap X = \emptyset$ then $D[B_X,v]$ is an all-zero vector and the result holds, so assume $B_X \cap C \neq \emptyset$. By Lemma \ref{lem:modshortcircuit}\eqref{it:msc3} there are an $x \in X$ and a circuit $C'$ with $v \in C'$ and $C'\subseteq (C\setminus X) \cup x$.

  Let $M' := M\contract (B\setminus B_X)$. Then $C' \cap E(M') = \{v,x\}$ is a circuit of $M'$. Hence all $2\times 2$ subdeterminants of $[I\ A][B_X,\{v,x\}]$ have to be 0, which implies that $A[B_X,v]$ is the all-zero vector or parallel to $[I\ A][B_X,x]$.
\end{proof}

\begin{proof}[Proof of Theorem \ref{thm:parcon}]
  Let $A_1$, $A_2$, $A$ be as in the theorem, and define $E := X_1 \cup X_2 \cup X \cup Y_1 \cup Y_2 \cup Y$. Suppose there exists a $Z\subseteq E$ such that $A[Z]$ is square, yet $\det(A[Z])\not\in\parf$. Assume $A_1, A_2, A, Z$ were chosen so that $|Z|$ is as small as possible.

  If $Z \subseteq X_i \cup Y_i \cup X \cup Y$ for some $i \in \{1,2\}$ then $A[Z]$ is a submatrix of $A_i$, a contradiction. Therefore we may assume that $Z$ meets both $X_1\cup Y_1$ and $X_2\cup Y_2$. We may also assume that $A[Z]$ contains no row or column with only zero entries, so either there are $x \in X_1\cap Z$, $y \in Y_1\cap Z$ with $A_{xy} \neq 0$ or $x \in X\cap Z$, $y \in Y_1 \cap Z$ with $A_{xy} \neq 0$.

  In the former case, pivoting over $xy$ leaves $D_X$, $D_2$, and $D_2'$ unchanged, yet by Lemma~\ref{lem:detpivot} $\det(A[Z])\in\parf$ if and only if $\det(A^{xy}[Z\setminus \{x,y\}])\in\parf$. This contradicts minimality of $|Z|$. Therefore $Z\cap X_1 = \emptyset$. Similarly, $Z\cap X_2 = \emptyset$.

  Define $X' := Z\cap X$. Now pick some $y \in Y_1$. Since $A[X',Y_1\cup Y]$ is obtained from $A[X,Y_1\cup Y]$ by deleting rows, it follows from Lemma~\ref{lem:modflatparallel}, applied to $M[I\ A_1]$, that the column $A[X',y]$ is either a unit vector (i.e. a column of an identity matrix) or parallel to $A[X',y']$ for some $y'\in Y$. In the first case, Lemma~\ref{lem:detpivot} implies again that $\det(A[Z])\in\parf$ if and only if $\det(A[Z\setminus \{x,y\}])\in\parf$, contradicting minimality of $|Z|$. In the second case, if $y' \in Z$ then $\det(A[Z]) = 0$. Otherwise we can replace $y$ by $y'$ without changing $\det(A[Z])$ (up to possible multiplication with some nonzero $p\in\parf$). It follows that $\det(A[Z]) = p' \det(A[Z'])$, where $Z'\subseteq X \cup Y \cup Y_2$, and $p'\in\parf\setminus\{0\}$. But $\det(A[Z'])\in\parf$, so also $\det(A[Z])\in\parf$, a contradiction.

  It remains to prove that $M[I\ A] = P_N(M_1,M_2)$. Suppose $\parf = (R,G)$, and let $I$ be a maximal ideal of $R$. Let $\phi:R\rightarrow R/I$ be the canonical ring homomorphism. For a square $\parf$-matrix $D$ we have $\det(D) = 0$ if and only if $\det(\phi(D)) = 0$. Hence $M[I\ A] = M[I\ \phi(A)]$. But $R/I$ is a field, so the result now follows directly from Brylawski's original theorem.
\end{proof}


The special cases $X = \emptyset$ and $X = \{p\}$ were previously proven by Semple and Whittle~\cite{SW96}.

\section{The need for 4-sums}\label{sec:4sums}
The core of the proof of Theorem \ref{thm:obstruct} will be a special matroid $M_{12} := M[I\ A_{12}]$, where 
\begin{align}
  A_{12} = \kbordermatrix{
        & d & e & f & 4 & 5 & 6\\
      a & 1 & 0 & 1 & 1 & 1 & 0\\
      b & 0 & -1& 1 & 1 & 0 & \alpha\\
      c & 1 & 1 & 0 & 0 & \alpha & -\alpha\\
      1 & 0 & 0 & 0 & 1 & 0 & 1\\
      2 & 0 & 0 & 0 & 0 & 1 & -1\\
      3 & 0 & 0 & 0 & 1 & 1 & 0
      }.
\end{align}

\begin{lemma}\label{lem:M12props}
	The following hold:
	\begin{enumerate}
		\item\label{M12:1} $M_{12}$ is near-regular;
		\item\label{M12:2} $M_{12}$ is internally 4-connected;
		\item\label{M12:3} $M_{12}$ is self-dual;
		\item\label{M12:4} $M_{12} \delete \{1,2,3,4,5,6\} \cong M(K_4)$;
		\item\label{M12:5} $M_{12} \contract \{a,b,c,d,e,f\} \cong M(K_4)$;
		\item\label{M12:6} No triad of $M_{12}\delete \{1,2,3,4,5,6\}$ is a triad of $M_{12}$. 
	\end{enumerate}
\end{lemma}

\begin{figure}[tbp]
\center
  \includegraphics{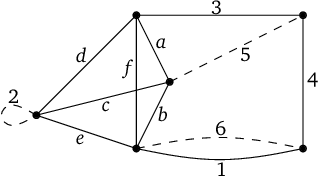}
\caption{Signed-graphic representation of $M_{12}$. Negative edges are dashed; positive edges are solid.}\label{fig:M12}
\end{figure}

We will omit the proofs, each of which boils down to a finite case check that is easily done on a computer and not too onerous by hand. Specifically,  for the first property one can either verify that $A_{12}$ is totally near-unimodular, or that $M_{12}$ contains none of the excluded minors for near-regular matroids (see Hall et al. \cite{HMZ09}). The latter approach is facilitated by observing that $M_{12}$ is the signed-graphic matroid associated with the signed graph illustrated in Figure \ref{fig:M12}. That graph can also be used to verify \eqref{M12:2}, by examining all edge-partitions $(A,B)$ that meet in two or three vertices. The remaining properties are readily extracted from the matrix $A_{12}$.

We will use the $M(K_4)$-restriction to create the generalized parallel connection of $M_{12}$ with $M(K_n)$. The following is well-known, but we will include the short proof.

\begin{lemma}\label{lem:Knconn}
	The matroid $M(K_n)$ is internally 4-connected.
\end{lemma}

\begin{proof}
	Fix an integer $n$, and suppose $(A,B)$ is a $3$-separation of $M(K_n)$ with $|A|, |B| \geq 4$. It follows that $n \geq 5$. Assume that $\rank(A) \geq \rank(B)$. Note that $\closure(A)$ and $\closure(B)$ induce complete subgraphs of $K_n$, and that these subgraphs meet in at most three vertices. It follows that, for some vertex $v$ of $K_n$, all edges incident with $v$ are in $A$, or all edges are in $B$. Assume the former. Then $\closure(A) = E(K_n)$, and therefore $\rank(A) = n-1$, and $\rank(B) = 2$. But then $B$ is a subset of a triangle of $K_n$, a contradiction.
\end{proof}

We need to show that in forming the generalized parallel connection we do not introduce unwanted 3-separations. The following lemma takes care of this.

\begin{lemma}\label{lem:no3seps}
	Let $M_1 = M(K_n)$ for some $n \geq 5$, and $M_2$ an internally 4-connected matroid such that there is a set $X = E(M_1)\cap E(M_2)$ with $N := M_1|X = M_2|X \cong M(K_4)$. Then $M := P_N(M_1,M_2)$ is a well-defined matroid. If no triad of $N$ is a triad of $M_2$ then $M$ is internally 4-connected.
\end{lemma}

\begin{proof}
	It is well-known (see \cite[Page 236]{oxley}) that $N$ is a modular flat of $M_1$. Hence $M = P_N(M_1,M_2)$ is well-defined. It remains to prove that $M$ is internally 4-connected. Suppose not. $M$ is obviously connected. Suppose $(A,B)$ is a 2-separation of $M$. By relabelling we may assume $|A\cap E(M_1)| \geq |B\cap E(M_1)|$. By Lemma \ref{lem:minorksep} we have that $(A\cap E(M_1), B\cap E(M_1))$ is 2-separating in $M_1$ (since $M_1$ is a restriction of $M$). But $M_1$ is 3-connected, so $|B\cap E(M_1)| \leq 1$. Similarly we have either $|A\cap E(M_2)| \leq 1$ or $|B\cap E(M_2)| \leq	1$. Since $|E(M_1) \cap E(M_2)| = 6$, the latter must hold. Hence $B = \{e,f\}$ for some $e \in E(M_1)\setminus E(N)$ and $f \in E(M_2)\setminus E(N)$. Since $E(M_1)$ and $E(M_2)$ are flats of $M$, we have $\rank_M(\{e,f\}) = 2$. Moreover $e \in \closure_M(E(M_1)\setminus e)$ and $f\in \closure_M(E(M_2)\setminus f)$, so $\{e,f\} \subseteq \closure_M(A)$. But then
	\begin{align}
		\rank_M(A) + \rank_M(B) - \rank(M) = \rank_M(B) = 2,
	\end{align}
	contradicting the fact that $(A,B)$ is a 2-separation.
	
	Next suppose that $(A,B)$ is a 3-separation of $M$ with $|A| \geq 4$ and $|B| \geq 4$. By relabelling we may assume $|A\cap E(M_1)| \geq |B\cap E(M_1)|$. By Lemma \ref{lem:minorksep} again, $(A\cap E(M_1),B\cap E(M_1))$ is 3-separating in $M_1$. Since $M_1$ is internally 4-connected, $|B\cap E(M_1)| \leq 3$. Define $T := B\cap E(M_1)$.
	
	We will show that $T\subseteq \closure_M(B\setminus T)$. Since $M_1$ has no cocircuits of size less than 4, we have $T\subseteq \closure_M(A)$. Therefore
	\begin{align}
		\rank_M(A\cup T) + \rank_M(B\setminus T) - \rank(M) & = \rank_M(A) + \rank_M(B\setminus T) - \rank(M)\notag\\
		 & \leq \rank_M(A) + \rank_M(B) - \rank(M) = 2.\label{eq:1}
	\end{align}
	If $|B\setminus T| \geq 2$ then it follows from 3-connectivity that equality holds in \eqref{eq:1}, so $\rank_M(B) = \rank_M(B\setminus T)$. If $|B\setminus T| = 1$ then $\rank_M(B\setminus T) = 1$ and we must have $\rank_M(B) = 2$. In that case $T$ is a triangle of $M_1$ and some element $e \in E(M_2)\setminus E(M_1)$ is in the closure of $T$. But no such element $e$ exists, since $E(M_1)$ is a flat of $M$.
	
	Note that $B\setminus T \subseteq E(M_2)$. Since $T\subseteq \closure_M(B\setminus T)$ and $E(M_2)$ is a flat of $M$, we have that $T\subseteq E(M_2)$. Hence $T\subseteq E(N)$, and $B\cap E(M_2) = B$. Since $(A\cap E(M_2), B\cap E(M_2))$ is 3-separating and $|B\cap E(M_2)| = |B| \geq 4$, we have $|A\cap E(M_2)| \leq 3$. But $|B\cap E(M_1)| \leq 3$, and therefore $E(N)\setminus B \subseteq A\cap E(M_2)$, from which it follows that $|A\cap E(M_2)| \geq 3$.
	
	Since no triad of $N$ is a triad of $M_2$, we must have that $A\cap E(M_2)$ is a triangle of $M_2$. Hence $B\cap E(N)$ is a triad of $N$. Now consider $(A\cap E(M_1), B\cap E(M_1))$ again. This partition of $M_1$ must be 3-separating, but $B\cap E(M_1)$ is not a triangle of $M_1$, and $M_1$ has no 3-element cocircuits. This contradiction completes the proof.	
\end{proof}

\begin{proof}[Proof of Theorem~\ref{thm:obstruct}]
  It suffices to prove the theorem for $G_1 = G_2 = K_n$, where $n \geq 5$. Label the edges of some $K_4$-restriction $N_1$ of $G_1$ by $\{a,b,c,d,e,f\}$, and define
  \begin{align}
  	M' := \left(P_{N_1}\left(M(G_1), M_{12}\right)\right)^*.
  \end{align}
  By Theorem \ref{thm:parcon}, $M'$ is near-regular, and by Lemma \ref{lem:no3seps}, $M'$ is internally 4-connected.

  Note that we still have $M'|\{1,2,3,4,5,6\} \cong M(K_4)$. Label the edges of some $K_4$-restriction $N_2$ of $G_2$ by $\{1,2,3,4,5,6\}$, and define
  \begin{align}
    M := P_{N_2}\left(M(G_2), M'\right).
  \end{align}
  By Theorem~\ref{thm:parcon}, $M$ is near-regular, and by Lemma \ref{lem:no3seps}, $M$ is internally 4-connected. The result follows.
\end{proof}

Matroid $M_{12}$ was found while studying the 3-separations of $R_{12}$. The unique 3-separation $(X,Y)$ of $R_{12}$ with $|X| = |Y| = 6$ is induced in the class of regular matroids. Pendavingh and Van Zwam had found, using a computer search for blocking sequences, that it is not induced in the class of near-regular matroids.

Unlike $R_{10}$ and $R_{12}$ in Seymour's work, the matroid $M_{12}$ by itself is quite inconspicuous. A natural class of near-regular matroids is the class of near-regular signed-graphic matroids. As indicated earlier, $M_{12}$ is a member of this class (see Figure~\ref{fig:M12}). The $K_4$-restriction is readily identified. $M_{12}$ is self-dual and has an automorphism group of size 6, generated by $(c,e) (d,f) (1,5) (3,6)$ and $(a,d) (b,e) (1,4) (2,3)$.

\section{Conjectures}\label{sec:conjectures}
While Theorem \ref{thm:obstruct} is a bit of a setback, we remain hopeful that a satisfactory decomposition theory for near-regular matroids can be found. First of all, the construction in Section~\ref{sec:4sums} employs only graphic matroids. In fact, it seems difficult to extend the $M(G_1)$-restriction of the 4-sum to some strictly near-regular matroid. The proof of Theorem~\ref{thm:obstruct} suggests the following construction:
\begin{definition}
  Let $M_1, M_2$ be matroids such that $E(M_1)\cap E(M_2) = X$, $N := M_1|X = M_2|X \cong M(K_k)$, and $M_1$ is graphic. Then the \emph{graph $k$-clique sum} of $M_1$ and $M_2$ is $P_N(M_1,M_2)\delete X$.
\end{definition}

Now we offer the following update of Conjecture~\ref{con:nregdecompfalse}:

\begin{conjecture}\label{con:nregdecomp}
  Let $M$ be a near-regular matroid. Then $M$ can be obtained from matroids that are signed-graphic, are the dual of a signed-graphic matroid, or are members of a finite set ${\cal C}$, by applying the following operations:
  \begin{enumerate}
    \item 1-, 2-, and 3-sums;
    \item Graph $k$-clique sums and their duals, where $k \leq 4$.
  \end{enumerate}
\end{conjecture}

Note that the work of Geelen et al.~\cite{GGW06b}, when finished, should imply a decomposition into parts that are bounded-rank perturbations of signed-graphic matroids and their duals. However, the bounds they require on connectivity are huge. Conjecture~\ref{con:nregdecomp} expresses our hope that for near-regular matroids specialized methods will give much more refined results.

As noted in the introduction, Seymour's Decomposition Theorem is not the only ingredient in the proof of Theorem \ref{cor:regpoly}. Another requirement is that the basic classes can be recognized in polynomial time. The following result suggests that this may not hold for the basic classes of near-regular matroids:

\begin{theorem}\label{thm:signedgraphnotpoly}
  Let $M$ be a signed-graphic matroid. Let $N$ be a matroid on $E(M)$ given by a rank oracle. It is not possible to decide if $M = N$ using a polynomial number of rank evaluations.
\end{theorem}

A matroid is \emph{dyadic} if it is representable over $\GF(p)$ for all primes $p > 2$. Since all signed-graphic matroids are dyadic (which was first observed by Dowling~\cite{Dow72}), this in turn implies that dyadic matroids are not polynomial-time recognizable. 

A proof of Theorem \ref{thm:signedgraphnotpoly}, analogous to the proof by Seymour~\cite{Sey81} that binary matroids are not polynomial-time recognizable, was found by Jim Geelen and, independently, by the first author. It involves ternary swirls, which have a number of circuit-hyperplanes that is exponential in the rank. To test if the matroid under consideration is really the ternary swirl, all these circuit-hyperplanes have to be examined, since relaxing any one of them again yields a matroid.

However, this family of signed-graphic matroids is not near-regular for all ranks greater than 3. 
Hence the complexity of recognizing near-regular signed-graphic matroids is still open. The techniques used by Seymour~\cite{Sey81} do not seem to extend, but perhaps some new idea can yield a proof of the following conjecture:

\begin{conjecture}
  Let ${\cal C}$ be the class of near-regular signed-graphic matroids. Then ${\cal C}$ is polynomial-time recognizable.
\end{conjecture}

In fact, we still have some hope for the following:

\begin{conjecture}
  The class of near-regular matroids is polynomial-time recognizable.
\end{conjecture}


\paragraph{Acknowledgements}We thank the anonymous referee for many useful suggestions. The third author thanks Rudi Pendavingh for introducing him to matroid theory in general, and to the problem of decomposing near-regular matroids in particular.

\newpage

\renewcommand{\Dutchvon}[2]{#1}
\bibliography{matbib2009}
\bibliographystyle{siam} 

\end{document}